\documentclass[11pt]{article}
\usepackage{latexsym}
\usepackage{theorem}
\usepackage{graphicx}
\usepackage{amsmath}
\usepackage{color}
\usepackage{mathrsfs}
\usepackage{xcolor}
\usepackage{amsfonts}
\usepackage{enumitem}
\usepackage[page]{appendix}
\usepackage{natbib}
\usepackage{pdf14}
\usepackage{soul}
\usepackage{hyperref} 
\usepackage{mathtools}
\DeclarePairedDelimiter{\floor}{\lfloor}{\rfloor}

\usepackage[margin=1in]{geometry}

\theorembodyfont{\rmfamily}
\newtheorem{theorem}{Theorem}
\newtheorem{lemma}[theorem]{Lemma}
\newtheorem{proposition}[theorem]{Proposition}

\newtheorem{definition}[theorem]{Definition}
\newenvironment{proof}[1][Proof]{\noindent\textbf{#1.} }{\ \rule{0.5em}{0.5em}}

\DeclareMathOperator*{\argmax}{arg\,max}

\theoremstyle{break}
\newtheorem{algorithm}[theorem]{Algorithm}

\newcommand{\al}{\alpha}

\graphicspath{{../fig/}}

\title{Computing Optimal Strategies for a Search Game in Discrete Locations}

\author{Jake Clarkson\thanks{Centre Inria d'Université Côte d'Azur, Sophia Antipolis, 06902, France, jake.clarkson@inria.fr} \and Kyle Y. Lin\thanks{Operations Research Department, Naval Postgraduate School, Monterey, CA 93943, kylin@nps.edu}}

\begin{document}

\maketitle

\begin{abstract}
\noindent
Consider a two-person zero-sum search game between a hider and a searcher.
The hider hides among $n$ discrete locations, and the searcher successively visits individual locations until finding the hider. 
Known to both players, a search at location $i$ takes $t_i$ time units and detects the hider---if hidden there---independently with probability $\al_i$, for $i=1,\ldots,n$.
The hider aims to maximize the expected time until detection, while the searcher aims to minimize it.
We present an algorithm to compute an optimal strategy for each player. We demonstrate the algorithm's efficiency in a numerical study, in which we also study the characteristics of the optimal hiding strategy.
\end{abstract}


\section{Introduction} \label{sec:intro}
Consider the following two-person zero-sum game $G$ studied in \cite{clarkson2022classical}. A hider chooses one of $n$ locations (henceforth boxes for conciseness) to hide in, and a searcher searches these boxes one at a time in order to find the hider.
A search in box $i$ takes a known $t_i > 0$ time units and will find the hider with known probability $\al_i \in (0, 1)$ if the hider is there, for $i=1,\ldots, n$. The searcher wants to minimize the expected total time until the hider is found, while the hider wants to maximize it.

The hider has $n$ pure strategies, each corresponding to a box to hide in. 
A mixed hiding strategy sees the hider choose a probability $p_i \in[0,1]$ with which to hide in box $i$, and is written by $\mathbf{p} \equiv (p_1,\ldots,p_n)$. Due to the possibility of overlook, the searcher may need to visit a box many times to find the hider; therefore, the search length can be arbitrarily long, and a pure search strategy must specify an indefinite, ordered sequence of boxes to search.




As first discovered by Blackwell and reported in \cite{Matula}, when the hider's mixed strategy is fixed, an optimal search strategy is simple to calculate.
\cite{Heller19} provides general methods for solving games where a best response of player 2 to any fixed player 1 strategy is easily obtained.
However, the algorithms of \cite{Heller19} require both players to have finite pure strategy sets, which is not the case for the searcher in $G$.

Indeed, the infinite number of pure search strategies means $G$ is semi-finite, and hence difficult to analyze. Building on previous work for special cases, \cite{clarkson2022classical} proves that there exist optimal strategies for both players, and that there exists an optimal search strategy that is a mixed strategy of at most $n$ sequences of boxes of a simple form.
While \cite{clarkson2022classical} offers several insightful properties of both players' optimal strategies, and presents a method to test whether a pair of strategies are optimal, it remains unclear how to compute each player's optimal strategy in general.



Despite being studied since the 1960s, solved cases of the search game are rare. Write $\mathbf{p}_0$ for the hiding strategy that hides in box $i$ with probability proportional to $t_i/\alpha_i$. Intuitively, the hiding strategy $\mathbf{p}_0$ is a reasonable choice because the hider prefers to hide in a box $i$ if its search time $t_i$ is large and its detection probability $\alpha_i$ is small.
When all the boxes are identical, $\mathbf{p}_0$, which hides in any box with probability $1/n$, is optimal for the hider by symmetry.
For the special case when $\alpha_i=0.5$ for all boxes, \cite{Ruckle:1991} presents a deterministic search sequence that is optimal for the searcher. 
For the special case where $n=2$ and $t_1=t_2=1$, \cite{GitRob1} identify some sufficient conditions on $\alpha_1$ and $\alpha_2$ for $\mathbf{p}_0$ to be optimal.
\cite{Ruckle:1991} further specializes to the case where $\alpha_2=1$, providing optimal strategies for both players.





Other than these few special cases, the search game $G$ remains unsolved.
Instead, the literature suggests several methods to estimate optimal strategies. For two boxes with unit search times, \cite{GitRob1} finds numerically that $\mathbf{p}_0$ is always a decent approximation to an optimal hiding strategy and is sometimes optimal. These conclusions are extended to $n$ unit-search-time boxes by \cite{GitRob2}. 
To compute an optimal hiding strategy, \cite{Bram1963} and \cite{GitRob2} both suggest a hill-climbing procedure, since the function to be optimized is concave. As for an optimal search strategy, \cite{Bram1963} and \cite{GitRob2} both propose considering optimal counter strategies to several perturbations of an optimal hiding strategy, the former suggesting solving the finite game formed with them. However, these approaches are rather arbitrary, with no guarantee of obtaining an optimal strategy to any specified accuracy.

In a technical report, \cite{lin2015robust} proposes an idea to iteratively compute lower and upper bounds for the value of the game $G$ by solving a new, larger finite game in each iteration. If the bounds converge, the solution to the finite game provides optimal strategies for $G$, but it is not clear if convergence is guaranteed.

This paper presents an algorithm that numerically computes optimal mixed strategies for both players to an arbitrary degree of accuracy limited only by computing power.
Our algorithm shares a similar spirit to that in \cite{lin2015robust}.
In each iteration, however, we use some new findings in \cite{clarkson2022classical} to construct bounds for the value of the game $G$ that are guaranteed to converge.



The simple case of $G$ with $\al_i=1$ for $i=1,\ldots,n$ is discussed in \cite{Heller19}. Due to the perfect detection, a pure search strategy is a permutation of $\{1,\ldots , n\}$, and therefore the search game is finite and easy to solve. The unique optimal hiding strategy is $\mathbf{p}_0$, but there are many optimal search strategies, detailed by \cite{Lidbetter13}. 

Whilst $G$ adds imperfect detection, a wider literature has extended this simple search game in other ways. \cite{Lidbetter13} and \cite{LidbLin17} add multiple hidden objects; the searcher wants to minimize the expected total search time to find all objects while the hider wants to maximize it. \cite{Kikuta90} adds a travel time between boxes, envisaging equally-spaced boxes lying on a straight line. The payoff to both players is total time both searching \emph{and} traveling until the object is found. \cite{Kikuta91} drops the equally-spaced assumption, and \cite{BastonKikuta:2015} considers boxes connected by a more general structure with direction-dependent travel times.


\cite{Efron:1964} studies the case of $t_i=1$ for $i=1,\ldots , n$ with a mobile hider who can switch boxes between searches. The searcher maximizes the probability of finding the hider before a known deadline, whilst the hider wants to minimize it. \cite{Subelman} considers the same objective, but, as in our search game $G$, with imperfect detection and an immobile hider. \cite{Lin2016} extend to the case where the deadline is not known by the searcher, showing that there exists a search strategy optimal for any deadline.

In \cite{GalCasas:2014}, the hider is \textit{prey} and the searcher a \textit{predator}, with the deadline interpreted as the predator tiring and needing rest. In this model,  the prey is always seen when the predator searches the correct box $j$, but capture only occurs with probability $\al_j$, since the prey will attempt to flee. The game ends when either the deadline expires, or there is a capture or escape. \cite{GalAlpCasas:2015} studies an extension where, until the deadline is reached, the game may continue after an escape, with the prey moving to another hiding location. \cite{GalAlpCasas:2019} extends further by allowing the chance of the fleeing prey being caught whilst relocating.

Rather than on discrete locations, search games played on a network or in subset of $\mathbb{R}^n$ have also been studied extensively in the literature. 
For example, \cite{BeckNew1970} study search on the unbounded real line and \cite{Isaacs1965} on any bounded, continuous set. \cite{Gal1980} considers a general, symmetric network, where the hider can hide at any point on any edge. Recent extensions to \cite{Gal1980} include a searcher with two speeds \citep{AlpernLid}, direction-dependent travel times down edges \citep{Alpern:2010, AlpernLid:2014}, and a searcher who must return the hider to a root node before the search completes \citep{Alpern11}.
For an overview on search games, please see \cite{garnaev2000search}, \cite{AlpernGal}, \cite{Alpern4auth}, and \cite{hohzaki2016search}.

The rest of the paper proceeds as follows.
Section~\ref{sec:prelim} recaps some earlier results and presents some new results that will be used to develop the algorithm.
Section~\ref{sec:algorithm} presents an algorithm that computes an upper bound and a lower bound for the optimal value, and proves that the bounds will converge.
Section~\ref{sec:numerical} presents a numerical study whose findings explain several puzzling phenomena in the search games of \cite{GitRob1} and \cite{Ruckle:1991}.
Section~\ref{sec:conclude} concludes.

\section{Preliminaries}
\label{sec:prelim}
Consider the search game $G$.
Recall a pure hiding strategy is a box in which to hide, so the hider's pure strategy space is $\{1, \ldots, n\}$.
A mixed hiding strategy is a probability vector $\mathbf{p} \equiv (p_1,\ldots, p_n) \in \Delta^n$, where $p_i$ is the probability that the hider hides in box $i$ and
\[
\Delta^n \equiv \left\{(p_1, \ldots, p_n) : p_i \in [0,1] \mbox{ for } i=1,\ldots, n \mbox{ and } \sum_{i=1}^n p_i = 1 \right\}.
\]
Recall a pure search strategy is an infinite, ordered list of boxes to search until the hider is found, which we now call a \emph{search sequence}.
The searcher's pure strategy space is therefore the infinite Cartesian product $\mathcal{C} \equiv \{1,2,\ldots,n\}^\infty$, and hence $G$ is a semi-finite two-person zero-sum game. Since the payoff---namely the total time to detection---is bounded by below by 0, standard results in \cite{Blackwell1954} affirm that $G$ has a value, $v^*$, and an optimal hiding strategy, which we denote by $\mathbf{p}^*$.
The uncountability of $\mathcal{C}$ creates difficulties on the searcher's side, but \cite{clarkson2022classical} proves that the searcher also has an optimal strategy.

We define a mixed search strategy as a non-negative function $\theta$ with domain $\mathcal{C}$ such that the set $\{\xi \in \mathcal{C} : \theta(\xi)>0\}$ is finite, and $\sum_{\xi \in \mathcal{C}} \theta(\xi) = 1$.
Under strategy $\theta$, the searcher plays search sequence $\xi \in \mathcal{C}$ with probability $\theta(\xi)$, and we say $\theta$ is a \emph{mixture} of those $\xi$ with $\theta(\xi)>0$. 
It is well known that any optimal search strategy $\theta^*$ is a mixture of search sequences which are best responses to any optimal hiding strategy $\mathbf{p}^*$.
\cite{Ross:1983} classifies a best response to be precisely a \textit{Gittins search sequence}, defined below.




\begin{definition} \label{def:GIP}
A \emph{Gittins search sequence} against a mixed hiding strategy $\mathbf{p} \equiv (p_1, \ldots, p_n)$ is an infinite, ordered list of boxes that meets the following rule. If $m_i \in \{0,1,2,\ldots\}$ searches have already been made of box $i$ during the search process, for $i = 1, \ldots , n$, then the next search is some box $j$ satisfying
\begin{equation} \label{eq:simpleGI}
j = \argmax_{i \in \{1, \ldots ,  n\}} \frac{p_i(1-\al_i)^{m_i}\al_i}{t_i}.
\end{equation}
The terms in \eqref{eq:simpleGI} we call \emph{Gittins indices}, since, by a comment of Kelly on \cite{Gittins79}, our search game with fixed $\mathbf{p}$ may be posed as a multi-armed bandit problem with an optimal index solution. A Gittins search sequence, therefore, always next searches a box with a maximal Gittins index. We write $\mathcal{C}_{\mathbf{p}}^{\rm B} \subset \mathcal{C}$ for the set of Gittins search sequences against hiding strategy $\mathbf{p}$.
\end{definition}

To summarize, we have the following result.

\begin{proposition} \label{prop:searcheroptcounter}
Any optimal search strategy $\theta^*$ is a mixture of sequences contained in $\mathcal{C}_{\mathbf{p^*}}^{\rm B}$ for every optimal hiding strategy $\mathbf{p}^*$. In words, a search sequence used in an optimal search strategy must be an optimal counter to every optimal hiding strategy.
\end{proposition}

If $|\mathcal{C}_{\mathbf{p}^*}^{\rm B}|=1$, then the sole search sequence in $\mathcal{C}_{\mathbf{p}^*}^{\rm B}$ is optimal. 
Yet, each time there is a tie for the maximal Gittins index in \eqref{eq:simpleGI}, the searcher is free to choose the order to search the tied boxes, so distinct sequences in $\mathcal{C}_{\mathbf{p}^*}^{\rm B}$ are spawned. Therefore, it is possible for $\mathcal{C}_{\mathbf{p}^*}^{\rm B}$ to be any size, even infinite, making it difficult to find an optimal mixture.

Write $\widehat{\mathcal{C}}_{\mathbf{p}}^{\rm B} \subseteq \mathcal{C}_{\mathbf{p}}^{\rm B}$ for those Gittins search sequences against $\mathbf{p}$ which always search tied boxes in the same order.
For example, if $n=4$, the Gittins search sequence in $\widehat{\mathcal{C}}_{\mathbf{p}}^{\rm B}$ corresponding to the ordering $(3,4,1,2)$ will break any tie between any subset of boxes using the order of preference 3, 4, 1, 2. 
Since there are $n!$ permutations of the boxes $\{1,\ldots , n\}$, we have $|\widehat{\mathcal{C}}_{\mathbf{p}}^{\rm B}|\leq n!$. A key result from \cite{clarkson2022classical} is the following.
\begin{theorem} \label{thm:OSSperm}
In the search game $G$, for any optimal hiding strategy $\mathbf{p}^*$, there exists an optimal search strategy which is a mixture of at most $n$ elements of $\widehat{\mathcal{C}}_{\mathbf{p}^*}^{\rm B}$.
\end{theorem}

Theorem \ref{thm:OSSperm} shows the searcher can focus their efforts on $\widehat{\mathcal{C}}_{\mathbf{p}^*}^{\rm B}$ which has size at most $n!$.
Therefore, if we can somehow guess $\mathbf{p}^*$ correctly, then we can obtain an optimal search strategy by solving a finite matrix game whose size is no larger than $n \times n!$. To formalize this idea, write $G_\mathcal{D}$ for the finite matrix game where the hider as usual can choose any of the $n$ boxes to hide in, but the searcher can choose a search sequence from only a finite subset $\mathcal{D}$ of $\mathcal{C}$.  
\cite{clarkson2022classical} uses Theorem~\ref{thm:OSSperm} to induce the following optimality test for the hider which requires only the solution to a finite matrix game.

\begin{proposition} \label{prop:subgamesoln}
Consider $\mathbf{p} \equiv (p_1,\ldots , p_n)$ with $p_i>0$ for $i=1,\ldots, n$, and the finite matrix game $G_\mathcal{D}$ with $\mathcal{D} = \widehat{\mathcal{C}}_{\mathbf{p}}^{\rm B}$.
Write $\theta$ for an optimal search strategy in $G_\mathcal{D}$.
The following three statements are equivalent.
\begin{enumerate}[label=(\roman*)]
\item $\mathbf{p}$ is optimal in $G_\mathcal{D}$; \label{enu:i}
\item $\theta$ is optimal in $G$; \label{enu:ii}
\item $\mathbf{p}$ is optimal in $G$. \label{enu:iii}
\end{enumerate}
\end{proposition}

Recall the hiding strategy $\mathbf{p}_0 \equiv (p_{0,1}, \ldots ,p_{0,n})$ discussed in Section \ref{sec:intro} with
\begin{equation*} 
p_{0,i} \equiv \frac{t_i/\alpha_i}{\sum_{j=1}^n t_j/\alpha_j}, \quad i=1,\ldots , n,
\end{equation*}
which is known to be optimal in several special cases, and is found numerically to be optimal by both \cite{GitRob1} and \cite{GitRob2} in many (but not all) unit-search-time problems. 
The hiding strategy $\mathbf{p}_0$ creates a tie between the Gittins indices of all $n$ boxes at the start of the search, giving the searcher no preference over which box to search first. With Proposition~\ref{prop:subgamesoln}, if $n$ is small, we can quickly test whether $\mathbf{p}_0$ is optimal and also compute an optimal search strategy if so. Proposition \ref{prop:subgamesoln} is put into practice in the numerical experiments of Section \ref{sec:numerical}.

If $\mathbf{p}_0$ is suboptimal, Proposition \ref{prop:subgamesoln} is less useful for finding an optimal hiding strategy, due to the infinite number of hiding strategies available. For this instance, we provide an algorithm in Section \ref{sec:algorithm} that estimates an optimal strategy for each player by successively computing tighter bounds on the value $v^*$. To develop the algorithm, we need some more results. We begin by reciting one more from \cite{clarkson2022classical}.

Define $u(i, \xi)$ as the expected time to detection if the hider hides in box $i$ and the searcher uses search sequence $\xi \in \mathcal{C}$.
In addition, define $u(i, \theta)$, $u(\mathbf{p}, \xi)$, and $u(\mathbf{p}, \theta)$ analogously where $\mathbf{p}$ is a mixed strategy for the hider and $\theta$ is a mixed strategy for the searcher.

\begin{proposition} \label{prop:pstar>0}
The following statements are true.
\begin{enumerate} [label={(\roman*)}]
\item If $\mathbf{p}^*\equiv(p_1^*,\ldots , p_n^*)$ is optimal for the hider, then $p_i^*>0$ for $i = 1, \ldots ,n$. \label{prop:pstar>01}
\item \label{prop:pstar>02} If $\theta^*$ is optimal for the searcher, then $u(i,\theta^*) = v^*$ for $i = 1,\ldots,n$.
\end{enumerate}
\end{proposition}

Proposition \ref{prop:pstar>0}\ref{prop:pstar>01} shows that $p_i^* \in (0,1)$ for $i=1,\ldots,n$. We next present two lemmas which we will combine to show that $p_i^*$ is bounded below by some strictly positive number, for $i=1, \ldots, n$.

\begin{lemma} \label{lemma:contradictpstar}
Consider some hiding strategy $\mathbf{p} \in \Delta^n$.
If there exists some box $i \in \{1, \ldots, n\}$ such that either $u(i, \xi) > v^*$ for all $\xi \in \mathcal{C}_{\mathbf{p}}^{\rm B}$ --- or $u(i, \xi) < v^*$ for all $\xi \in \mathcal{C}_{\mathbf{p}}^{\rm B}$ --- then $\mathbf{p}$ is not optimal for the hider. 
\end{lemma}

\begin{proof}
We use proof by contradiction. 
Suppose that $\mathbf{p}$ is optimal for the hider and $u(i, \xi) > v^*$ for all $\xi \in \mathcal{C}_{\mathbf{p}}^{\rm B}$ for some box $i \in \{1, \ldots, n\}$.
Write $\theta^*$ for an optimal search strategy.
According to Proposition \ref{prop:searcheroptcounter}, $\theta^*$ mixes only search sequences in $\mathcal{C}_{\mathbf{p}}^{\rm B}$, so $u(i, \theta^*)>v^*$, which contradicts Proposition \ref{prop:pstar>0}\ref{prop:pstar>02}.
Hence, $\mathbf{p}$ cannot be optimal for the hider. 
The case $u(i, \xi)<v^*$ can be proved with a similar argument.
\end{proof}

\begin{lemma}
The value $v^*$ of the search game $G$ is bounded above by $\sum_{i=1}^n t_i/\al_i$ and below by $\max_{i \in \{1, \ldots , n\}} t_i/\al_i$.
\label{le:UB}
\end{lemma}
\begin{proof}
To derive the upper bound, consider a variation of the search game in which the hider is free to move between boxes after each unsuccessful search, solved by \cite{Norris1962}. The value of this search game with a mobile hider is
\begin{equation} \label{eqn:mobilevalue}
M \equiv \sum_{i=1}^n \frac{t_i}{\al_i}.
\end{equation}
A mobile hider has more options, so \eqref{eqn:mobilevalue} is an upper bound on the value $v^*$ of $G$.

To prove the lower bound, suppose the hider reveals that they will hide in box $j$. The optimal counter for the searcher is to search box $j$ repeatedly until finding the hider, leading to an expected search time of $t_j/\al_j$, which is a lower bound on $v^*$. 
Since this argument applies to any box $j$, the maximal value of $t_i/\al_i$ over all boxes $i \in \{1, \ldots , n\}$ is the best such lower bound on $v^*$.
\end{proof}

\bigskip

We next use Lemmas \ref{lemma:contradictpstar} and \ref{le:UB} to derive a lower bound for $p_i^*$. To begin, define
\begin{equation}
m_i \equiv \left\lfloor \frac{M}{t_i} \right\rfloor + 1;
\label{eq:m_i}
\end{equation}
for $i=1,\ldots,n$, where $M$ is the upper bound on $v^*$ in \eqref{eqn:mobilevalue} from Lemma \ref{le:UB}.
In other words, making $m_i$ searches in box $i$ requires more than $M$ time units. The following bound is based on the idea that if $p_i$ is too small, there must be some box $j$ that the searcher needs to search at least $m_j$ times before searching box $i$ for the first time. 
Consequently, $M$ time units have passed before box $i$ is searched for the very first time.
We can then invoke Lemma \ref{lemma:contradictpstar} to show that $p_i$ cannot be optimal.

\begin{proposition}
\label{prop:simplebound}
For $i=1,\ldots,n$, write
\begin{equation} \label{eqn:delta_i}
c_i \equiv \sum_{j=1, j \neq i}^n \frac{t_j}{\al_j (1-\al_j)^{m_j-1}} \qquad \text{and} \qquad \eta_i \equiv \frac{t_i/\al_i}{t_i/\al_i + c_i},
\end{equation}
where $m_j$ is defined in \eqref{eq:m_i}.
Any optimal hiding strategy $(p_1^*, \ldots, p_n^*)$ must have $p_i^* \geq \eta_i$, for $i=1,\ldots, n$.
\end{proposition}
\begin{proof}
Without loss of generality, we prove $p_1^* \geq \eta_1$.
To begin, write $\mathbf{p} = (p_1,\ldots,p_n)$ for some hiding strategy such that 
\begin{equation} \label{eqn:boxjmjsearch}
\frac{p_j \al_j(1-\al_j)^{m_j-1}}{t_j}>\frac{p_1 \al_1}{t_1},
\end{equation}
for some box $j \neq 1$.
Using the Gittins index formula in \eqref{eq:simpleGI}, we see that any Gittins search sequence $\xi \in \mathcal{C}_{\mathbf{p}}^{\rm B}$ searches box $j$ at least $m_j$ times before searching box 1 for the first time.
According to the definition of $m_j$ in \eqref{eq:m_i}, more than $M$ time units have elapsed before box 1 is searched for the first time, so we must have $u(1, \xi) > M \geq v^*$ for all $\xi \in \mathcal{C}_{\mathbf{p}}^{\rm B}$.
Consequently, we can use Lemma~\ref{lemma:contradictpstar} to conclude that $\mathbf{p}$ is not optimal for the hider.

Now write $(\eta_1, p'_2, \ldots, p'_{n})$ for the unique solution to the following system of linear equations:
\begin{align}
\frac{p'_j \al_j (1-\al_j)^{m_j-1}}{t_j} &= \frac{\eta_1 \al_1}{t_1}, \qquad j = 2, \ldots , n;
\label{eqn:lineareqn1}
\\
\eta_1 + \sum_{j=2}^{n} p'_j &= 1.
\label{eqn:lineareqn2}
\end{align}
It is straightforward to verify that solving the preceding yields $\eta_1$ defined in \eqref{eqn:delta_i} with $i=1$. 
If $p_1 < \eta_1$ for some hiding strategy $\mathbf{p}\equiv (p_1, \ldots , p_n)$, then due to \eqref{eqn:lineareqn2} there must exist some $j \neq 1$ such that $p_j > p'_j$.
It then follows from \eqref{eqn:lineareqn1} that $p_j$ and $p_1$ satisfy \eqref{eqn:boxjmjsearch}, so $\mathbf{p}$ is not optimal for the hider.
Consequently, the optimal hiding strategy must have $p_1^* \geq \eta_1$, which concludes the proof.
\end{proof}

The bounds in Proposition \ref{prop:simplebound} are easy to calculate, but need not be tight.
The significance of Proposition \ref{prop:simplebound}, however, is that, by choosing $\delta_i < \eta_i$ (for example, let $\delta_i \equiv 0.99 \, \eta_i$), for $i=1,2,\ldots, n$, we can define a closed set
\begin{equation} \label{eqn:pwithdeltacons}
\Delta^n_+ \equiv \left\{\mathbf{p} : \sum_{i=1}^n p_i = 1;  \; p_i \geq \delta_i \text{ for } i=1,\ldots,n\right\},
\end{equation}
such that any optimal hiding strategy $\mathbf{p}^*$ must lie in the \textit{interior} of $\Delta^n_+$. 

Our final result of this section shows that the payoff $u$ is always bounded when the searcher plays a Gittins search sequence against some hiding strategy in $\Delta^n_+$.

\begin{proposition} \label{prop:bounded_cond_est}
Consider a hiding strategy $\mathbf{p}\in \Delta^n_+$. For any search sequence $\xi \in \mathcal{C}_{\mathbf{p}}^{\rm B}$, we have $u(i,\xi)$ bounded above for $i = 1,\ldots,n$.
\end{proposition}
\begin{proof}
Consider the search sequence $\xi'$ that repeats the cycle of searches $(1,\ldots , n)$ indefinitely. 
Writing $T \equiv \sum_{i=1}^n t_i$, we can compute
\begin{align*}
u(i,\xi') &= \sum_{j=1}^i t_j + T \sum_{k=1}^{\infty} (1-\alpha_i)^k = \sum_{j=1}^i t_j + \frac{1-\alpha_i}{\alpha_i} \; T,
\end{align*}
so $u(\mathbf{p},\xi') = \sum_{i=1}^n p_i u(i,\xi')$ is finite.

For any $\xi \in \mathcal{C}_{\mathbf{p}}^{\rm B}$, we must have 
\[
p_i u(i,\xi) < \sum_{i=1}^n p_i u(i,\xi) = u(\mathbf{p},\xi) \leq u(\mathbf{p},\xi'),
\]
where the first inequality follows since $p_i > 0$ and $u(i, \xi)>0$ for $i=1,\ldots,n$, and the last inequality follows since $\xi$ is an optimal counter to $\mathbf{p}$.
Consequently, for $i=1,\ldots,n$,
$$u(i,\xi) < \frac{u(\mathbf{p},\xi')}{p_i} \leq \frac{u(\mathbf{p},\xi')}{\delta_i},$$
where the last inequality follows since $\mathbf{p} \in \Delta^n_+$.
In other words, $u(i,\xi)$ is bounded above by $u(\mathbf{p},\xi') / \delta_i$.
\end{proof}

\section{Computing Optimal Strategies} \label{sec:algorithm}
This section presents an iterative algorithm to compute the value $v^*$ of the game $G$, and an optimal strategy for each player.
Each iteration of the algorithm involves solving a finite matrix game in which the searcher can choose only from a finite set of search sequences.
The value of each finite matrix game is an upper bound for $v^*$ due to the searcher's finite pure strategy set.
The hiding strategy optimal in the current iteration induces an optimal counter---a Gittins search sequence---which is added to the searcher's finite strategy set to best complement the searcher's repertoire, before the algorithm proceeds to the next iteration.
Through iterations, the expanding sets of search sequences produce a decreasing sequence of upper bounds for $v^*$, which we will prove converges on $v^*$. 
Consequently, we can compute a strategy for each player that guarantees a payoff arbitrarily close to $v^*$ by solving the algorithm's finite game after some number of iterations.

To facilitate presenting the algorithm, define
\[
u(\mathbf{p}) \equiv \min_\xi u(\mathbf{p}, \xi)
\]
as the expected search time if the hider uses mixed strategy $\mathbf{p} \in \Delta^n$ and the searcher uses any optimal counter $\xi \in \mathcal{C}^B_\mathbf{p}$---namely, any Gittins search sequence against $\mathbf{p}$.
The value of $G$ can be written as
\[
v^* = \max_{\mathbf{p} \in \Delta_+^n} u(\mathbf{p}).
\]
We now present this algorithm.



\begin{algorithm}
\label{al:minExpCost}
%
%

\begin{enumerate}
\item \label{step:setup} 
Initialize $L=0$ and $U=\infty$ as a lower bound and an upper bound for $v^*$, respectively.
Pick $\epsilon > 0$ so that when the algorithm stops we have $U/L-1 < \epsilon$.

\item \label{step:initialize}
Recall the hiding strategy $\mathbf{p}_0 \equiv (p_{0,1}, \ldots ,p_{0,n})$ with
\begin{equation} \label{eqn:p0formula}
p_{0,i} \equiv \frac{t_i/\alpha_i}{\sum_{j=1}^n t_j/\alpha_j}, \quad i=1,\ldots , n.
\end{equation}
Compute the $n$ Gittins search sequences of $\widehat{\mathcal{C}}_{\mathbf{p}_0}^{\rm B}$ which break all ties using $n$ preference orderings $(1,2,\ldots,n)$, $(2,3,\ldots,n, 1), \ldots$, $(n,1,\ldots,n-1)$.
Form a set $\mathcal{D}_1$ with these $n$ search sequences and proceed to iteration $k=1$.

\item \label{step:solveandcheck}



In iteration $k$, consider the constrained finite matrix game $G_k$ where the searcher's pure strategy set is $\mathcal{D}_k$ and the hider is constrained to mixed strategies in $\Delta^n_+$ defined in \eqref{eqn:pwithdeltacons}.
Solve the following linear program, which finds an optimal hiding strategy $\mathbf{p}_k$ and the value $v_k$ of $G_k$. 
\begin{align}
\max_{\mathbf{p}, v} &\quad v & \nonumber\\
\rm{s.t.} & \quad v \leq \sum_{i=1}^n p_iu(i,\xi) \quad \text{for all} \quad \xi \in \mathcal{D}_k; \label{eq:C-constraint_alg} \\
& \quad  \quad p_i \geq \delta_i, \; \; i=1,\ldots, n; \label{eq:P-constraint_alg} \\
& \quad \sum_{i=1}^n p_i = 1. \nonumber
\end{align}

\item \label{step:update_bounds}
Update $U_k \leftarrow v_k$. Compute a Gittins search sequence $\xi_k$ against $\mathbf{p}_k$, and use $\xi_k$ to calculate $u(\mathbf{p}_k)$.
Update $L_k \leftarrow  u(\mathbf{p}_k)$.

\item \label{step:check_stop}
If (i) $U_k/L_k - 1 < \epsilon$ and (ii) none of the constraints in \eqref{eq:P-constraint_alg} are binding, then go to step 6; otherwise, update $\mathcal{D}_{k+1} \leftarrow \mathcal{D}_k \cup \{\xi_k\}$ and go to step \ref{step:solveandcheck} for iteration $k+1$.

\item \label{step:calc_search_final}
Calculate the searcher's optimal strategy $\theta$ in $G_k$.
Output $\mathbf{p}_k$, $\theta$, and $U_k$.
\end{enumerate}
\end{algorithm}


If Algorithm \ref{al:minExpCost} terminates in iteration $k$, we have $U_k/L_k -1< \epsilon$.
In step~\ref{step:calc_search_final}, the output hiding strategy guarantees the hider an expected time to detection of at least $L_k$, and the output search strategy guarantees the searcher an expected time to detection of at most $U_k$.
Because $\Delta^n_{+}$ contains any optimal hiding strategy $\mathbf{p}^*$ in $G$, $(\mathbf{p}^*,v^*)$ is a feasible solution to the optimization problem in Step \ref{step:solveandcheck}; therefore, the value $v_k=U_k$ of $G_k$ is an upper bound on $v^*$. 
Furthermore, we have $L_k=u(\mathbf{p}_k)\leq u(\mathbf{p}^*)=v^*$ by definition.
Consequently, we have $L_k \leq v^* \leq U_k$.

The rest of this section consists of three parts. Section \ref{sec:alg_rationale} paints a general picture of how Algorithm \ref{al:minExpCost} works with the help of a diagram for $n=2$. Section \ref{sec:alg_proof} proves that Algorithm \ref{al:minExpCost} terminates, whilst Section \ref{sec:alg_discuss} offers further discussion on Algorithm \ref{al:minExpCost}. 

\subsection{Rationale of Algorithm~\ref{al:minExpCost}} \label{sec:alg_rationale}
The rationale of Algorithm \ref{al:minExpCost} is best understood via an example with $n=2$ boxes, where any mixed hiding strategy $\mathbf{p} = (p, 1-p)$ can be delineated by a single number $p \in [0,1]$ which represents the probability of hiding in box 1.
Figure~\ref{fig:cuttingplane} demonstrates Algorithm \ref{al:minExpCost} in action. Each slanted, straight line represents the expected time to detection for a search sequence as the hiding strategy $p$ varies in $[0,1]$. The function $u(p)$ is the lower envelope of the set of all search sequences, hence, a concave function in $p$, as indicated by the bold curve in Figure~\ref{fig:cuttingplane}.
Note that each straight line representing the expected time to detection of a Gittins search sequence is tangential to the curve $u(p)$ only at those hiding strategies for which it is a Gittins search sequence against; at any other point it lies strictly above $u(p)$. 
Since several neighboring hiding strategies can share the same Gittins search sequence, $u(p)$ is piecewise linear.
We seek to determine $v^* \equiv \max_{p \in [0,1]} u(p)$.

\begin{figure}[h!]
\begin{center}
\includegraphics[scale=0.6]{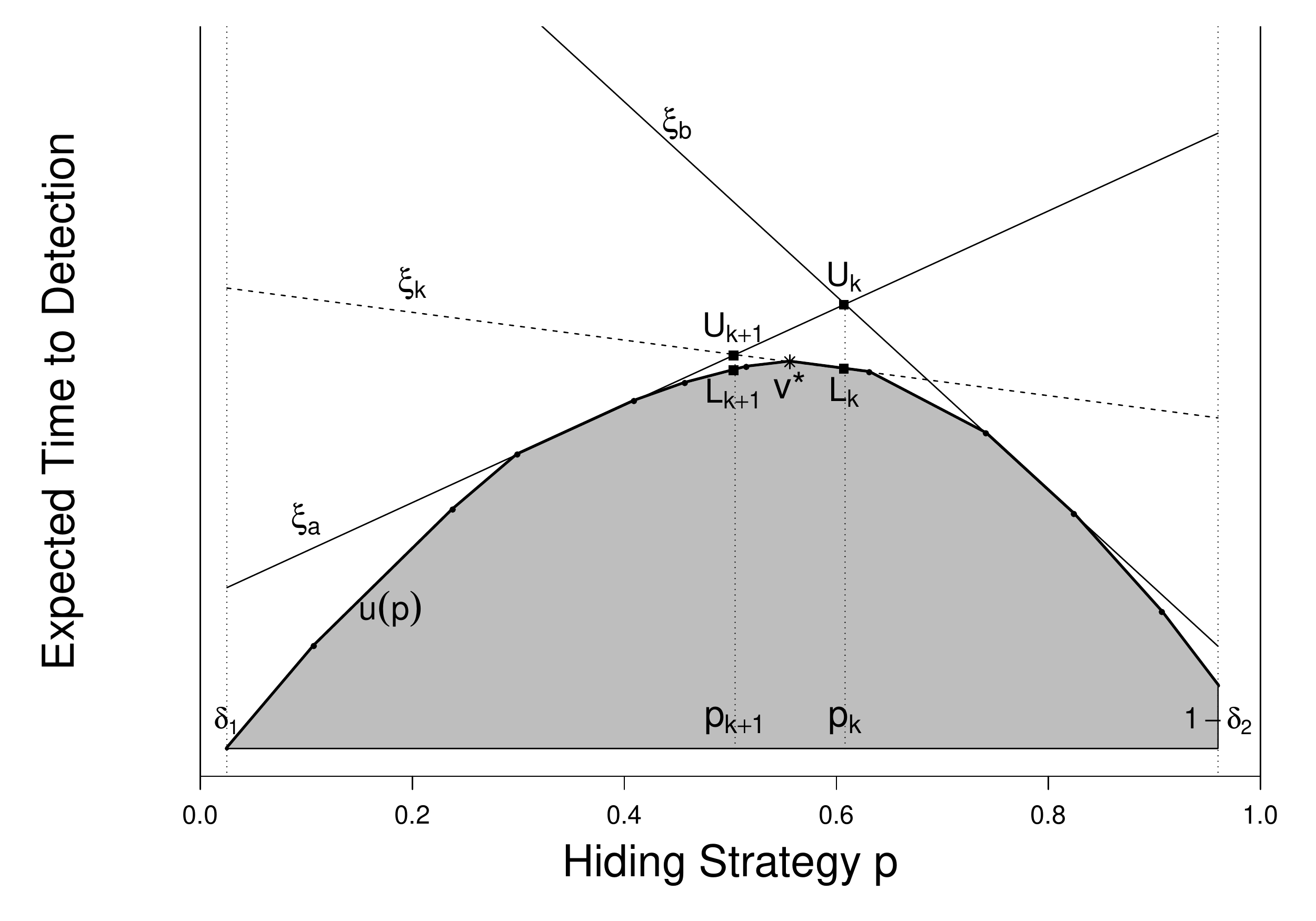}
\caption{A snapshot of Algorithm \ref{al:minExpCost} with $n=2$ for iterations $k$ and $k+1$.}
\label{fig:cuttingplane}
\end{center}
\end{figure}

Consider iteration $k$, when there are $n+k-1$ search sequences in the searcher's set $\mathcal{D}_k$ forming the constraints in \eqref{eq:C-constraint_alg}. In \eqref{eq:P-constraint_alg}, Algorithm \ref{al:minExpCost} also adds constraints $p\geq \delta_1$ and $p\leq 1-\delta_2$, shown by the vertical dotted lines in Figure \ref{fig:cuttingplane}.
If these additional constraints are non-binding at iteration $k$, as is the case in Figure~\ref{fig:cuttingplane}, then $G_k$ is a standard, finite matrix game with an optimal search strategy mixing at most two search sequences, which we represent by the two solid straight lines $\xi_a$ and $\xi_b$.

By optimally mixing $\xi_a$ and $\xi_b$, the searcher guarantees that the expected time to detection is no more than $U_k$, which is an upper bound for $v^*$.
The expected time to detection when the hider plays $p_k$ and the searcher plays a corresponding Gittins search sequence $\xi_k$ is a lower bound for $v^*$, namely $L_k = u(p_k)$.
Therefore, the dashed straight line representing $\xi_k$ touches the curve $u(p)$ at $L_k$ when $p=p_k$.

By adding $\xi_k$ to $\mathcal{D}_k$ to form $\mathcal{D}_{k+1}$ in Step \ref{step:check_stop}, the searcher may include $\xi_k$ in their mixed strategy in $G_{k+1}$, which is advantageous since it cuts off the previous optimal point $(p_k,U_k)$. Since $D_k \subseteq D_{k+1}$, this cut results in a new, tighter upper bound $U_{k+1}$; consequently, the sequence $\{U_k\}$ is weakly decreasing. As seen in Figure~\ref{fig:cuttingplane}, it is possible that $L_{k+1}$ computed in iteration $k+1$ is smaller than $L_k$ computed in iteration $k$, yet Theorem \ref{thm:alg_cnvg} will prove that both sequences $\{U_k\}$ and $\{L_k\}$ converge to the value of the game $v^*$.

For arbitrary $n$, the process behind Algorithm \ref{al:minExpCost} is identical, but the straight lines representing search sequences become hyperplanes in $n$-dimensional space. In particular, for $n=3$, search sequences are represented by planes, whose lower envelope becomes a dome.

\subsection{Algorithm~\ref{al:minExpCost} Terminates} \label{sec:alg_proof}
In this section, we prove that Algorithm~\ref{al:minExpCost} terminates for any $\epsilon >0$ so we can always approximate $v^*$ to arbitrary precision.

\begin{theorem} \label{thm:alg_cnvg}
Algorithm~\ref{al:minExpCost} terminates for arbitrary $\epsilon > 0$.
\end{theorem}

\begin{proof}
Recall that for $\xi \in \mathcal{C}$ we have
$$u(\mathbf{p},\xi)\equiv \sum_{i=1}^n p_i u(i,\xi),$$
which is a linear function in $\mathbf{p}\equiv(p_1,\ldots,p_n)$. 
Further, recall that
\[
u(\mathbf{p}) \equiv \min_\xi  u(\mathbf{p},\xi)
\]
is the lower envelope of these functions linear in $\mathbf{p}$, thus is a concave function in $\mathbf{p}$. 

\begin{figure}[h!]
\begin{center}
\includegraphics[scale=0.6]{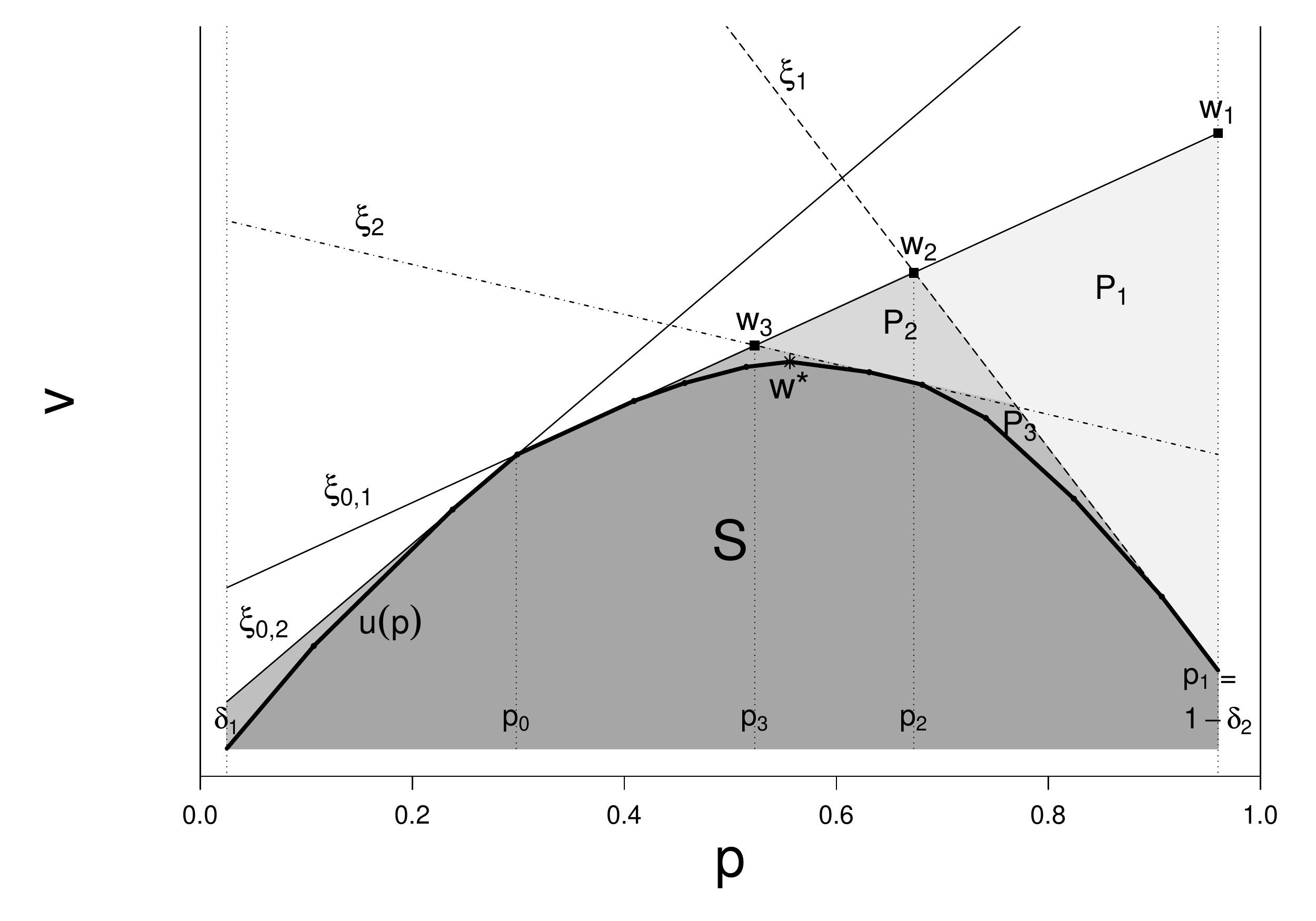}
\caption{Example of the beginning of Algorithm \ref{al:minExpCost} for $n=2$ if $p_0$ is not optimal.}
\label{fig:cuttingplane_proof}
\end{center}
\end{figure}

Consider the following optimization problem:
\begin{align}
\max_{\mathbf{p}, v} &\quad v & \nonumber\\
\rm{s.t.} & \quad v \leq u(\mathbf{p}); \label{eq:C-constraint} \\
& \quad p_i \geq \delta_i, \; \; i=1,\ldots, n; \label{eq:P-constraint} \\
& \quad \sum_{i=1}^n p_i = 1,\;  v\geq 0. \label{eq:sum1_v-cons}
\end{align}
The set of feasible solutions $S$ can be considered as a subset of $\mathbb{R}^n$ with elements of the form $(p_1,\ldots,p_{n-1},v)$ due to $\sum_{i=1}^n p_i=1$.
Since $\mathbf{p}=(p_1,\ldots,p_n)$, however, throughout the proof we shall write the $n$-dimensional vector $(p_1,\ldots,p_{n-1},v)$ as $(\mathbf{p},v)$ for ease of notation.

Since $\Delta^n_+$ in \eqref{eqn:pwithdeltacons} contains all optimal hiding strategies in $G$, solving the optimization problem above finds $\mathbf{w}^* = (\mathbf{p}^*,v^*)$, where $\mathbf{p}^*$ is optimal for the hider in $G$ and $v^* = \max_\mathbf{p} u(\mathbf{p}) = u(\mathbf{p}^*)$ is the value of $G$. It is possible that the hider has multiple $\mathbf{p}^*$ leading to multiple $\mathbf{w}^*$, but $v^*$ is unique.


To help demonstrate the idea of the proof, we will use an example shown in Figure \ref{fig:cuttingplane_proof} with $n=2$ boxes, where any mixed hiding strategy $\mathbf{p}=(p,1-p)$ can be delineated by $p \in [0,1]$.
The $x$-axis corresponds to $p$ and the $y$-axis corresponds to $v$, with each point identified by its coordinates $(p, v)$. The function $u(p)$ is the piecewise linear curve shown in bold.
The solution space $S$ is the area shaded in the darkest shade of grey, bounded above by $u(p)$, to the left and right by $\delta_1$ and $1-\delta_2$, and below by $v=0$. The point $\mathbf{w}^*\equiv(p^*,v^*)$ is the point in $S$ with the largest $v$ coordinate, which lies on the curve $u(p)$.

Write $\xi_{0,1},\ldots, \xi_{0,n}$ for the $n$ Gittins search sequences against $\mathbf{p}_0$ which form $\mathcal{D}_1$ in Step \ref{step:initialize} of Algorithm \ref{al:minExpCost}. The following inequalities define an intersection of half spaces in $\mathbb{R}^n$ with coordinates $(\mathbf{p}, v)$, a convex polytope which we call $P_1$. 
\begin{align}
& v \leq \sum_{i=1}^n p_iu(i,\xi_{0,j}), \; \; j=1,\ldots,n; \label{eqn:P1_vcons}\\
&  p_i \geq \delta_i, \; \; i=1,\ldots,n; \label{eqn:P1_pcons} \\
& \sum_{i=1}^n p_i = 1, \; v \geq 0. \label{eqn:P1_sum1_v-cons}
\end{align}
The constraints in \eqref{eqn:P1_pcons} and \eqref{eqn:P1_sum1_v-cons} are equivalent to those in \eqref{eq:P-constraint} and \eqref{eq:sum1_v-cons}.
Because
$$u(\mathbf{p})\equiv \min_{\xi\in \mathcal{C}} \sum_{i=1}^n p_iu(i,\xi) \leq \sum_{i=1}^n p_iu(i,\xi_{0,j}), \; \;  j=1,\ldots, n,$$
the constraints in \eqref{eqn:P1_vcons} are less restrictive than those in \eqref{eq:C-constraint}. It follows that the feasible solution space $S$ to the aforementioned optimization problem with solution $\mathbf{w}^*$ is contained in the polytope $P_1$. In Figure~\ref{fig:cuttingplane_proof}, $\xi_{0,1}$ and $\xi_{0,2}$ denote the hyperplanes induced by \eqref{eqn:P1_vcons}, and $P_1$ comprises all four shades of grey, which is the area enclosed by 5 straight lines, namely $\xi_{0,1}$, $\xi_{0,2}$, $p=\delta_1$, $p=1-\delta_2$, and $v=0$.

In each iteration of Algorithm \ref{al:minExpCost}, we add a new inequality of the same form to those already in \eqref{eqn:P1_vcons} to create a smaller polytope containing $S$. The result is a sequence of polytopes satisfying $P_1 \supset P_2 \supset \cdots \supset S$. In Figure~\ref{fig:cuttingplane_proof}, $P_2$ comprises the three darkest shades of grey, and $P_3$ comprises the two darkest shades of grey.

In iteration $k$, Algorithm \ref{al:minExpCost} finds a hiding strategy $\mathbf{p}_k$ in Step~\ref{step:solveandcheck} corresponding to the point with largest $v$ coordinate in polytope $P_k$, a point which we label $\mathbf{w}_k=(\mathbf{p}_k,v_k)$. The points $\mathbf{w}_1, \mathbf{w}_2$ and $\mathbf{w}_3$ are marked in Figure~\ref{fig:cuttingplane_proof}. Since $P_1 \supset P_2 \supset \cdots \supset S$, we also have the set of inequalities
\begin{align}
v_1 \geq v_2 \geq \cdots \geq v^*.
\label{eq:t_monotone}
\end{align}
Because $\{v_k\}$ is a weakly decreasing sequence bounded below by $v^*$, it must have a limit, denoted by $v_\infty \equiv \lim_{k \rightarrow \infty} v_k$.
Clearly $v_\infty \geq v^*$; we will show that $v_\infty = v^*$.

It is trivially true that the points $\{\mathbf{w}_k\}$ are all in $P_1$.
Because $P_1$ is a compact set, it follows that $\{\mathbf{w}_k\}$ has at least one limit point, so $\{\mathbf{w}_k\}$ has a converging subsequence. Pick an arbitrary converging subsequence and denote it by $\{\mathbf{w}_k: k \in \mathcal{K}\}$. Write $\mathbf{w}_\infty$ for the limit of the subsequence $\{\mathbf{w}_k: k \in \mathcal{K}\}$.

Consider $\mathbf{w}_k=(\mathbf{p}_k,v_k)$ in iteration $k \in \mathcal{K}$ of Algorithm \ref{al:minExpCost}. Step \ref{step:check_stop} adds a Gittins search sequence $\xi_k$ against $\mathbf{p}_k$ to the searcher's repertoire, which results in a hyperplane with equation
\begin{equation}
v=\sum_{i=1}^n p_{i}u(i,\xi_k), \label{eq:cut_first}
\end{equation}
creating polytope $P_{k+1}$ from $P_k$ and cutting off the point $\mathbf{w}_k$.
In Figure \ref{fig:cuttingplane_proof}, the hyperplane added to cut off $\mathbf{w}_1=(p_1,v_1)$ and create polytope $P_2$ is a Gittins search sequence against $p_1=1-\delta_2$ shown by the dashed line $\xi_1$, which intersects $\xi_{0,1}$ at $\mathbf{w}_2=(p_2,v_2)$. In the next iteration, $\mathbf{w}_2$ is cut off by adding a Gittins search sequence $\xi_2$ against $p_2$, which creates polytope $P_3$.

Since $\xi_k$ is a Gittins search sequence against $\mathbf{p}_k$, the point $(\mathbf{p}_k,u(\mathbf{p}_k))$ lies on the hyperplane in \eqref{eq:cut_first}, so \eqref{eq:cut_first} may also be expressed as
\begin{equation}
v=u(\mathbf{p}_k)+\sum_{i=1}^n (p_i-p_{i,k}) u(i,\xi_k),
\label{eq:cut}
\end{equation}
where $\mathbf{p}_k \equiv (p_{1,k},\ldots , p_{n,k})$.

If $k \in \mathcal{K}$, $k' \in \mathcal{K}$, and $k' > k$, then $\mathbf{w}_{k'}=(\mathbf{p}_{k'},v_{k'})$ is in polytope $P_{k+1}$. Since the hyperplane in \eqref{eq:cut} bounds $P_{k+1}$, we have
\begin{align}
v_{k'} \leq u(\mathbf{p}_k) + \sum_{i=1}^n (p_{i,k'}-p_{i,k}) u(i,\xi_k).
\label{eq:cut_leq}
\end{align}
As $k$ tends to infinity in $\mathcal{K}$, we must have $p_{i,k'} - p_{i,k} \rightarrow 0$ for $i=1,\ldots , n$ because $\{\mathbf{w}_k: k \in \mathcal{K}\}$ is a converging sequence. Since $\xi_k$ is a Gittins search sequence against $\mathbf{p}_k \in \Delta^n_+$, we have $u(i,\xi_k)$ bounded above for $i=1,\ldots,n$, according to Proposition \ref{prop:bounded_cond_est}. Therefore, we must have $(p_{i,k'} - p_{i,k}) u(i,\xi_k) \rightarrow 0$ for $i=1,\ldots , n$ as $k$ tends to infinity in $\mathcal{K}$.
In addition, we have $v_{k'} \rightarrow v_\infty$ and $u(\mathbf{p}_k) \rightarrow u(\mathbf{p}_\infty)$, so by \eqref{eq:cut_leq} we must have
\begin{equation}
v_\infty \leq u(\mathbf{p}_\infty),
\label{eq:v_infty}
\end{equation}
By the definition of $\mathbf{p}^*$ we have $u(\mathbf{p}_{\infty}) \leq u(\mathbf{p}^*)=v^*$. Together with \eqref{eq:t_monotone} and \eqref{eq:v_infty}, we conclude that $v^*=v_{\infty}=u(\mathbf{p}_{\infty})$. Consequently, the sequences of upper and lower bounds $\{U_k\}$ and $\{L_k\}$ both tend to $v^*$, so condition (i) in Step~\ref{step:check_stop} of Algorithm~\ref{al:minExpCost} will eventually be met.

Now consider condition (ii) in Step~\ref{step:check_stop} of Algorithm~\ref{al:minExpCost}. Since $u(\mathbf{p}_{\infty})=v^*$, we have $\mathbf{p}_{\infty} \equiv (p_{1,\infty},\ldots , p_{n,\infty})$ optimal for the hider. By Proposition \ref{prop:simplebound}, we therefore have $p_{i,\infty} > \delta_i$ for $i=1,\ldots , n$. Since the sequence $\{\mathbf{p}_k: k \in \mathcal{K}\}$ converges to $\mathbf{p}_{\infty}$ elementwise, it follows that condition (ii) must also be eventually met. The proof is completed.
\end{proof}

\subsection{Further Discussion on Algorithm~\ref{al:minExpCost}} \label{sec:alg_discuss}
In this section, we further discuss a few steps of Algorithm \ref{al:minExpCost}.

In Step \ref{step:initialize}, we use the hiding strategy $\mathbf{p}_0$ in \eqref{eqn:p0formula} to generate Gittins search sequences in $\widehat{\mathcal{C}}_{\mathbf{p}_0}$ to form the starting set $\mathcal{D}_1$. Recall that $\mathbf{p}_0$ creates a tie between all $n$ Gittins indices in \eqref{eq:simpleGI} at the start of the search, giving the searcher no preference over which box to search first. 
\cite{GitRob1} and \cite{GitRob2} numerically find that $\mathbf{p}_0$ often serves as a decent approximation to the optimal hiding strategy, in which case Gittins search sequences against $\mathbf{p}_0$ will be useful for the searcher.

As $n$ increases, it may become less computationally practical to form $\mathcal{D}_1$ with all $n!$ elements of $\widehat{\mathcal{C}}_{\mathbf{p}_0}$.
Further, even for small $n$, there is no numerical evidence that forming $\mathcal{D}_1$ with more than $n$ elements of $\widehat{\mathcal{C}}_{\mathbf{p}_0}$ would speed up convergence.
In Algorithm \ref{al:minExpCost}, we choose $n$ elements of $\widehat{\mathcal{C}}_{\mathbf{p}_0}$ by cycling the preference ordering $(1,\ldots, n)$ as described in step \ref{step:initialize} to include a variety of tie-breaking strategies against $\mathbf{p}_0$.

In Step \ref{step:solveandcheck}, it is critical to include the constraints $p_i \geq \delta_i$, for $i=1,\ldots, n$, when solving the matrix game.
As seen in Figure~\ref{fig:cuttingplane_proof}, when $n=2$ and $p_0$ is not optimal, since $u(p_0) < \max_{p}u(p)$ and $u(p)$ is concave, the tangents $\xi_{0,1}$ and $\xi_{0,2}$ which form $\mathcal{D}_1$ must slope in the same direction. Therefore, without the $\delta_i$ constraints, $\mathbf{w}_1$ would always correspond to the hiding strategy $(1,0)$ with which the hider hides in box 1 with probability 1.
The unique Gittins search sequence against $(1,0)$ never searches box 2, and would hence result in $u(2, \xi_1) = \infty$.
Adding $\xi_1$ to $\mathcal{D}_1$ does not help the searcher at all and thus prohibits the algorithm from moving forward.
å††The choice of $\delta_i$ in the constraints restricts the hider to the set $\Delta_{+}^n$ which still contains all hiding strategies optimal in $G$, so does not prevent the termination of Algorithm \ref{al:minExpCost}.



In Step \ref{step:update_bounds}, the bounds are updated.
As discussed in Section \ref{sec:alg_rationale}, since $D_k \subseteq D_{k+1}$, the sequence $\{U_k\}$ is weakly decreasing.
Yet, suppose we have $U_k=U_{k+1}$ for some iteration $k$. In this case, since the set of hiding strategies do not change from iteration $k$ to iteration $k+1$, the searcher must have a strategy optimal in $G_{k+1}$ that is available in $G_k$. Therefore, $\mathbf{p}_k$, optimal for the hider in $G_k$, is also optimal in $G_{k+1}$. It follows that $\mathbf{p}_k$ guarantees the hider an expected time to detection of $U_{k+1}$ in $G_{k+1}$. Since a Gittins search sequence $\xi_k$ against $\mathbf{p}_k$ is available to the searcher in $G_{k+1}$, $U_{k+1}$ cannot be greater than $u(\mathbf{p}_{k}) \leq v^*$. Since $U_{k+1}$ is an upper bound on $v^*$, we must have $U_k=U_{k+1}=v^*$. We conclude that, in fact, the sequence $\{U_k\}$ is \emph{strictly} decreasing until we have found $v^*$.

In Step \ref{step:check_stop}, we check two conditions before terminating the algorithm.
In particular, condition (ii) checks if the constraints in \eqref{eq:P-constraint_alg} hinder the hider's ability to seek a better strategy in $G_k$.
If none of these constraints are binding, then we know $\mathbf{p}_k$ is also optimal in $G_k$ if these constraints are removed, so $\mathbf{p}_k$ also guarantees $U_k$ to an unconstrained hider.

\section{Numerical Study} \label{sec:numerical}
This section presents several numerical experiments to demonstrate Algorithm \ref{al:minExpCost} and evaluate the performance of $\mathbf{p}_0$ defined in \eqref{eqn:p0formula} as a heuristic strategy for the hider.

To generate search times and detection probabilities for a set of $n$ boxes, we draw
\begin{equation} \label{eqn:acyclicdraw}
\alpha_i \sim U(\alpha_l,\alpha_u), \quad t_i \sim U(1,5), \quad i=1,\ldots,n.
\end{equation}
We study four schemes based on different values of $\alpha_l,\alpha_u \in (0,1)$, as seen in Table \ref{tab:dpstvar}.
\begin{table}
\caption{Sample schemes used throughout the numerical study by values of $\alpha_l$ and $\alpha_u$ used in \eqref{eqn:acyclicdraw}.} \label{tab:dpstvar}
\begin{center}  
\begin{tabular}{|c| c|} 
 \hline
Sample Scheme & $[\alpha_l,\alpha_u]$ \\
\hline
Varied & $[0.1,0.9]$ \\
Low    & $[0.1,0.5]$ \\
Medium & $[0.3,0.7]$ \\
High   & $[0.5,0.9]$ \\
\hline
\end{tabular}
\end{center}
\end{table}
Search games with $n=2$, $3$, $5$ and $8$ boxes will be investigated. To account for increased variation within a search game as $n$ increases, for each value of $n$ and sample scheme, we study $n \times 1000$ search games. Results are presented in Section \ref{sec:numerres}, but we first discuss our method to calculate conditional expected times to detection based on the hider's location under a Gittins search sequence.

\subsection{Calculating Expected Time to Detection} \label{sec:numercalcs}
The expected value of any nonnegative-valued random variable $X$ can be calculated by $E[X] = \int_0^\infty P\{X>x\} \, dx$. Using this formula, the expected time to detection if the searcher uses a search sequence $\xi$ and the hider hides in box $i \in \{1, \ldots,n\}$ can be calculated by
\begin{equation} \label{eqn:genv_i(xi)}
u(i,\xi) = \tau_i(1,\xi) + \lim_{R \rightarrow \infty} \sum_{r=1}^{R} (1-\alpha_i)^{r} [\tau_i(r+1,\xi)-\tau_i(r,\xi)],
\end{equation}
where $\tau_i(r,\xi)$ is the time at which the $r$th search of box $i$ is made under $\xi$. 

If $\xi$ is a Gittins search sequence against some hiding strategy $\mathbf{p}$, the terms $\tau_i(r,\xi)$ are determined by the Gittins indices in \eqref{eq:simpleGI} and the rule used by $\xi$ to break ties between indices.
In order to calculate $\tau_i(r,\xi)$ for a $\xi$ which uses a \emph{specific} tie-breaking rule, the searcher needs to properly recognize a tie between these Gittins indices.
Comparing indices directly, however, does not yield reliable results because the indices are encoded as floating-point numbers. However, we will see in the following that we can apply Algorithm \ref{al:minExpCost} to search problems generated using \eqref{eqn:acyclicdraw} without such concerns.

First, note that under \eqref{eqn:acyclicdraw}, for any distinct $i,j \in \{1,\ldots, n\}$, we draw $\alpha_i$ and $\alpha_j$ from a continuous uniform distribution. Therefore, the event $\log(1-\alpha_i)/\log(1-\alpha_j) \in \mathbb{Q}$ has probability 0, so
\begin{equation} \label{eqn:acyclic}
(1-\al_i)^{x}\neq (1-\al_j)^{y}
\end{equation}
is satisfied almost surely for any strictly positive integers $x$ and $y$.
In step \ref{step:initialize} of Algorithm \ref{al:minExpCost}, we need to evaluate $n$ Gittins search sequences against $\mathbf{p}_0$ which break ties between indices using a fixed preference ordering, such as $(1,2,\ldots,n)$.
Write $\xi_\sigma \in \widehat{\mathcal{C}}_{\mathbf{p}_0}^{\rm B}$ for the search sequence that uses preference ordering $\sigma$. At the beginning of the search, all $n$ indices are tied, so the first $n$ searches of $\xi_\sigma$ will correspond to the order of $\sigma$. Yet due to \eqref{eqn:acyclic}, we will not encounter any ties from the $n$th search onwards, so we can reliably calculate the terms $\tau_i(r,\xi_\sigma$) by comparing floating-point indices directly. In step \ref{step:update_bounds} of Algorithm \ref{al:minExpCost}, we also need to evaluate $\xi_k \in \mathcal{C}^{\rm B}_{\mathbf{p}_k}$ where $\mathbf{p}_k$ is a solution to the linear program in step \ref{step:solveandcheck}. Yet, because we can take \emph{any} $\xi_k \in \mathcal{C}^{\rm B}_{\mathbf{p}_k}$ in step \ref{step:update_bounds}, it is inconsequential whether ties between indices are recognized at any point in the search; comparing floating-point indices will always yield $\tau_i(r,\xi_k)$ for \emph{some} $\xi_k \in \mathcal{C}^{\rm B}_{\mathbf{p}_k}$.

To compute both $u(i, \xi_\sigma)$ and $u(i, \xi_k)$, first note that the partial sum of the first $R$ terms in \eqref{eqn:genv_i(xi)} provides a lower bound.
To obtain an upper bound, after the first $R$ searches in box $i$, adopt a search sequence that visits box $i$ at fixed intervals less frequently than any Gittins search sequence.
We increase $R$ until the ratio between the upper and lower bound is within $1 + 10^{-10}$.
See Appendix \ref{append:numerdetails_acyclic} for details.

\subsection{Numerical Results} \label{sec:numerres}
For each generated search game, we first look to test the optimality of $\mathbf{p}_0$ using Proposition \ref{prop:subgamesoln}, which involves solving the finite game $G_\mathcal{D}$ where the searcher is restricted to the set of pure strategies $\mathcal{D} \equiv\widehat{\mathcal{C}}^{\rm B}_{\mathbf{p}_0}$. By Proposition \ref{prop:subgamesoln}, $\mathbf{p}_0$ is optimal in $G$ if and only if $\mathbf{p}_0$ is optimal in $G_{\mathcal{D}}$. Since $G_{\mathcal{D}}$ may have multiple optimal hiding strategies, to determine the optimality of $\mathbf{p}_0$ in $G_{\mathcal{D}}$, we compare $v^*_\mathcal{D}$, the value of $G_\mathcal{D}$, to $u(\mathbf{p}_0)$, the expected time to detection when the hider plays $\mathbf{p}_0$ and the searcher plays any search sequence in $\mathcal{C}^{\rm B}_{\mathbf{p}_0}$. In principle, $\mathbf{p}_0$ is optimal in $G_{\mathcal{D}}$ if and only if $u(\mathbf{p}_0)$ and $v^*_\mathcal{D}$ are equal. 
For the purpose of this numerical study, we accept that $\mathbf{p}_0$ is optimal if $|u(\mathbf{p}_0) - v^*_\mathcal{D}| / v^*_\mathcal{D} < 10^{-9}$.
Table \ref{tab:results} presents the percentage of search games in which $\mathbf{p}_0$ is optimal for different sample schemes for $n=2, 3, 5$.
We do not use Proposition \ref{prop:subgamesoln} to test the optimality of $\mathbf{p}_0$ for $n=8$, since $|\widehat{\mathcal{C}}^{\rm B}_{\mathbf{p}_0}|\in\{1,\ldots,n!\}$, so the test could require the generation of $8!=40320$ search sequences.

Next, for each search game where Proposition \ref{prop:subgamesoln} finds $\mathbf{p}_0$ to be suboptimal, we run Algorithm \ref{al:minExpCost} to estimate the value $v^*$ and optimal strategies. Recall that Algorithm~\ref{al:minExpCost} terminates when the ratio of the upper and lower bound on $v^*$ is within $1+\epsilon$. Table \ref{tab:sseqsconveps} reports the mean and 95th percentile of the number of iterations needed for convergence under the \emph{varied} sample scheme for $\epsilon=10^{-3}, 10^{-6}$. Since step \ref{step:initialize} of Algorithm \ref{al:minExpCost} initializes $\mathcal{D}_1$ with $n$ search sequences, and step \ref{step:check_stop} adds one search sequence for the following iteration, in iteration $k$ there are $n+k-1$ search sequences in $\mathcal{D}_k$.

\begin{table}[htb!]
\caption{The mean (95th percentile) of the number of iterations required for convergence of Algorithm \ref{al:minExpCost}, for the \textit{varied} sample scheme. 
} 
\label{tab:sseqsconveps}
\begin{center}
\begin{tabular}{|c|c c|}
\hline 
$n$ & $\epsilon=10^{-3}$ & $\epsilon=10^{-6}$ \\ 
\hline
2  & 4.47 (5) & 6.63 (9)   \\
3  & 10.3 (13) & 15.9 (21)\\
5  & 28.7 (36) & 44.8 (58) \\
8  & 73.1 (92) & 113 (144) \\
\hline
\end{tabular}
\end{center}
\end{table}

We next assess the quality of $\mathbf{p}_0$ as a heuristic for the hider, extending the investigation in \cite{GitRob2} to arbitrary search times and a wider range of detection probabilities.
Table \ref{tab:results} shows the decrease from $v^*$ to $u(\mathbf{p}_0)$ as a percentage of $v^*$ for the sample schemes in Table \ref{tab:dpstvar}, with $v^*$ either deduced to be equal to $u(\mathbf{p}_0)$ by Proposition \ref{prop:subgamesoln}, or otherwise computed by Algorithm \ref{al:minExpCost} with $\epsilon =  10^{-6}$.

\begin{table}[htb!]
\small
\caption[caption]{The mean and percentiles of $u(\mathbf{p}_0)$ as percentage below optimum, and the percentage of search games in which $\mathbf{p}_0$ is optimal, for the sample schemes in Table \ref{tab:dpstvar}. 
}
\label{tab:results}
\begin{center}
\begin{tabular}{|c|c| c c c c|}
\hline 
$n$ & Metric &Varied & Low & Medium & High \\ 
\hline
2 & Mean    & 0.322 &  0.0734  & 0.0581  & 0.0357  \\
&95th Percentile & 1.43 &  0.291 & 0.363 & 0.213  \\
& \% $\mathbf{p}_0$ optimal    & 43.0 & 29.6 & 64.0 & 87.0 \\
\hline
3& Mean    & 0.537  &  0.0992   & 0.0524  & 0.0135  \\
&95th Percentile     & 1.72 &  0.31  & 0.301 & 0.0401  \\
& \% $\mathbf{p}_0$ optimal   & 21.4  & 12.7 & 55.7 & 91.7 \\
\hline
5& Mean    & 0.741 &  0.128  & 0.0441& 0.0012   \\
&95th Percentile    & 1.77 &  0.319  & 0.211 & 0 \\
& \% $\mathbf{p}_0$ optimal    & 7.06&  4.28 &44.4	&97.5  \\
\hline
8  & Mean  & 0.882	&0.148	&0.0335	&0.00003  \\
&95th Percentile      & 1.78	&0.303	&0.147	&0 \\
\hline
\end{tabular}
\end{center}
\end{table}

As seen in Table \ref{tab:results}, $\mathbf{p}_0$ generally performs well as a hiding heuristic for a range of $n$ and, for smaller $n$, often achieves optimality. With the \textit{varied} sample scheme, for each $n$, in 95\% of games $\mathbf{p}_0$ is within 1.78\% of optimality. Therefore, if the hider cannot run Algorithm \ref{al:minExpCost} to estimate $\mathbf{p}^*$, the easily-calculated $\mathbf{p}_0$ performs well as a heuristic. Table \ref{tab:results} shows that the optimality and performance of $\mathbf{p}_0$ depends strongly on both $n$ and the sample scheme. We next explain some patterns observed from Table~\ref{tab:results}.

\subsubsection{Patterns As the Sampling Scheme Varies}
Recall that $\mathbf{p}_0$ equates all $n$ Gittins indices at the start of the search. \cite{Norris1962} shows that if the hider was free to change boxes after every unsuccessful search, it is optimal for the hider to choose a new box according to $\mathbf{p}_0$, independent of previous hiding locations. In other words, it is optimal for a mobile hider to keep the Gittins indices equal throughout the search process. In our search game, the hider hides just once at the start of the search, so it is impossible for the hider to maintain equality of the evolving indices in \eqref{eq:simpleGI}.
Intuitively, the best the hider can do is hide with probability $\mathbf{p}^*$ such that, when the searcher follows a Gittins search sequence against $\mathbf{p}^*$, the indices in \eqref{eq:simpleGI} are as close to being equal as possible throughout the search. 
However, since the probability that the hider remains undetected decreases as time passes, it is more important for the hider to achieve equality in \eqref{eq:simpleGI} earlier in the search rather than later, explaining why $\mathbf{p}_0$ is, in general, a reasonable heuristic for the hider.


The preceding argument explains the following patterns in Table \ref{tab:results}. We see an improvement in performance of $\mathbf{p}_0$ in the \textit{high} sample scheme compared to the \textit{medium} sample scheme compared to the \textit{low} sample scheme, because the larger the detection probabilities, the sooner the hider is likely to be detected, and hence equality in \eqref{eq:simpleGI} near the start of the search takes even more importance. 

We also see an improvement in the performance of $\mathbf{p}_0$ in the \textit{medium} sample scheme, with its narrow range of detection probabilities, compared to the \textit{varied} sample scheme. To explain this phenomenon, we make the following connection to \cite{Clarkson:2020}, where the searcher knows the strategy of the hider, but has a choice between two search modes when searching any box. 
In our search game, the hider chooses $\mathbf{p}$ to make the search last as long as possible, which involves balancing maximizing uncertainty about their location and forcing the searcher into boxes with ineffective search modes.
\cite{Clarkson:2020} introduces two measures of the effectiveness of the search mode $(\alpha_i,t_i)$ of a box $i$. 
The first, called the \emph{immediate benefit}, is measured by $\alpha_i/t_i$. 
The larger the immediate benefit of box $i$, the greater the detection probability per unit time when box $i$ is searched. The second, called the \emph{future benefit}, is measured by
\begin{equation} \label{eqn:fbenefit}
\frac{-\log(1-\alpha_i)}{t_i}.
\end{equation}
If $p_i=p_j$ and box $i$ has a larger future benefit than box $j$, then an unsuccessful search of box $i$ gains more information per unit time about the hider's location than an unsuccessful search of box $j$.
Whilst $\mathbf{p}_0$ takes the immediate benefit of the $n$ boxes' search modes into account by hiding in box $i$ with probability proportional to $t_i/\alpha_i$, the future benefit is ignored by $\mathbf{p}_0$. 

In the game studied in \cite{Norris1962}, the hider may move between boxes after every unsuccessful search, so the game resets after every failed search.
Consequently, $\mathbf{p}_0$ is optimal since information gained by the searcher about the hider's location through an unsuccessful search is useless and hence the future benefit does not apply.
In our search game, however, gaining information about the hider's fixed location enables the searcher to make better box choices later in the search. Therefore, the hider should be dissuaded from hiding in boxes with a large future benefit, as the information-gain advantages of their search modes will benefit the searcher. Since $\mathbf{p}_0$ does not take future benefit into account, the larger the variation in future benefit between the $n$ boxes, the worse $\mathbf{p}_0$ performs. With the \textit{varied} sample scheme, there is more opportunity for such variation, so $\mathbf{p}_0$ performs worse here than in the narrower \textit{medium} sample scheme.

\subsubsection{Patterns As the Number of Boxes Varies}
As the number of boxes $n$ increases, the search is generally expected to last longer, so achieving equality in \eqref{eq:simpleGI} at the very start of the search by using $\mathbf{p}_0$ is less important. Further, the more boxes there are, the greater the uncertainty in the hider's location and hence the more valuable information about the hider's location becomes. Therefore, as discussed in \cite{Clarkson:2020}, the future benefit (ignored by $\mathbf{p}_0$) becomes more important as $n$ grows. Both of these factors contribute to the performance of $\mathbf{p}_0$ degrading with $n$ in the general \emph{varied} sample scheme.

However, recall that the smaller the variation in future benefit between the $n$ boxes, the better $\mathbf{p}_0$ will perform. In the narrow \emph{medium} sample scheme, this decrease in future benefit variation is the most significant factor as $n$ increases, leading to a slight improvement in the performance of $\mathbf{p}_0$ as $n$ increases in this case. Further, the size of the detection probabilities also has an effect. When they are low, adding a box to the game leads to a much bigger increase in the expected duration of the search than when they are high, explaining why, as $n$ increases, we see a greater improvement in the performance of $\mathbf{p}_0$ in \emph{high} compared to \emph{medium}, and also why the performance of $\mathbf{p}_0$ worsens for \emph{low}.

\subsection{Future Benefit for Two-Box Problems}
In this section, for $n=2$, we examine how future benefit affects the difference between $\mathbf{p}_0\equiv (p_0,1-p_0)$ and the optimal hiding strategy $\mathbf{p}^*\equiv (p^*,1-p^*)$. 

\cite{GitRob1} studied two-box search games with $\alpha_1 < \alpha_2$ and unit search times, noting that whenever $\mathbf{p}_0$ was suboptimal, $p^*$ was greater than $p_0$, but found no reason for this observation. We believe this phenomenon is explained by future benefit. Since $\alpha_1<\alpha_2$ and $t_1=t_2=1$, the future benefit in \eqref{eqn:fbenefit} at any $\mathbf{p}$ is greater for box 2 than box 1. Whilst $\mathbf{p}_0$ considers immediate benefit, it ignores future benefit, explaining why the hider, who wants the searcher to spend more time in boxes with inefficient search modes, may prefer to hide in box 1 with a probability greater than $p_0$.

To demonstrate this effect, we conduct an additional numerical study with $n=2$. The two boxes are drawn using \eqref{eqn:acyclicdraw}, then relabelled so box 1 has the lower future benefit in \eqref{eqn:fbenefit}. 
We generate 5,000 such search games. In the 3,049 where $p^*\neq p_0$, we found $p^* > p_0$ in 3,001 cases and $p^* < p_0$ in only 48. In other words, in the vast majority of cases where $p_0$ is suboptimal, the hider chooses box 1 (that with a smaller future benefit) with probability greater than $p_0$.


In addition, \cite{Ruckle:1991} solves a two-box game with $t_1=t_2=\alpha_2=1$ and a sole parameter $\alpha_1 \equiv \alpha \in (0,1)$. 
For this problem, \cite{Ruckle:1991} shows any Gittins search sequence against an optimal hiding strategy $p^*$ always searches box 1 until there is a tie between the two Gittins indices on the $h$th search for some $h \in \{1,2,\ldots\}$. In other words, $p_0$ is optimal for the hider if and only if $h=1$; otherwise, $p^*>p_0$. \cite{Ruckle:1991} shows that as $\alpha$ decreases, $h$ increases, so $p^*$ increases; see Table \ref{tab:Ruckle}.

We offer the following explanation. For any $\alpha \in (0,1)$, the future benefit in \eqref{eqn:fbenefit} is greater for box 2 than for box 1, so we always have $p^*>p_0$. Further, as $\alpha$ decreases, the future benefit of box 2 stays the same whilst the future benefit of box 1 decreases, which explains the growing difference between $p^*$ and $p_0$.

\begin{table}[h!]
\caption{For the two-box game with $t_1=t_2=\alpha_2=1$ and $\alpha_1 \equiv \alpha \in (0,1)$, the first search, $h$, before which the indices are equal by the value of $\alpha$.} \label{tab:Ruckle}
\begin{center}  
\begin{tabular}{|c| c|} 
 \hline
Value of $h$ & Range of $\alpha$ \\
\hline
1 & $[0.618,1]$ \\
2    & $[0.382,0.618]$ \\
3    & $[0.276,0.382]$ \\
\hline
\end{tabular}
\end{center}
\end{table}



\subsection{Cyclic Games} \label{sec:cyclic_games}
In real-life applications, the searcher will estimate detection probabilities $\alpha_1, \ldots, \alpha_n$ using a mixture of expert opinion or historical data. In this subsection, we show that if the searcher chooses their estimates such that
\begin{equation} \label{eqn:cyclic}
(1-\al_1)^{x_1}=\cdots = (1-\al_n)^{x_n}
\end{equation}
for some coprime, positive integers $x_1, \ldots, x_n$, then we may simplify the calculations of the conditional expected times to detection which are required both to test the optimality of $\mathbf{p}_0$ using Proposition \ref{prop:subgamesoln} or estimate optimal strategies using Algorithm \ref{al:minExpCost}. Under \eqref{eqn:cyclic}, after $x_i$ searches of box $i$ for $i=1,\ldots,n$, the posterior probability vector on the hider's location returns to the initial $\mathbf{p}$, so the search game has reset itself. Therefore, we call search games satisfying \eqref{eqn:cyclic} \emph{cyclic} search games.

Both step \ref{step:initialize} of Algorithm \ref{al:minExpCost} and the optimality test in Proposition \ref{prop:subgamesoln} require the evaluation of $\xi_\sigma$, namely a Gittins search sequence against $\mathbf{p}_0$ which breaks ties between Gittins indices using a fixed preference ordering $\sigma$. As discussed in Section \ref{sec:numercalcs}, since $\mathbf{p}_0$ ties all $n$ indices, the order of $\sigma$ determines the first $n$ searches of $\xi_\sigma$.
Due to \eqref{eqn:cyclic}, in a cyclic game, after the first $n$ searches we can reliably compare indices and hence reliably recognize ties by keeping track of the number of searches that $\xi_\sigma$ has performed in each box---as opposed to comparing floating-point indices directly (see Appendix \ref{append:numerdetails_cyclic} for details).
Also by \eqref{eqn:cyclic}, in its first $\sum_{i=1}^n x_i$ searches, any Gittins search sequence against $\mathbf{p}_0$ searches box $i$ exactly $x_i$ times, for $i=1,\ldots, n$. At that point, \eqref{eqn:cyclic} shows that all $n$ indices are once again tied, so the problem has reset itself.
Consequently, $\xi_\sigma$ will repeat the same cycle of $\sum_{i=1}^n x_i$ searches indefinitely, which leads to a closed form for $u(i, \xi_\sigma)$, $i=1,\ldots,n$; see Appendix \ref{append:numerdetails_cyclic} for details.

In fact, after an initial transient period, the aforementioned cycle of $\sum_{i=1}^n x_i$ searches will repeat indefinitely under \emph{any} Gittins search sequence $\xi$ against \emph{any} hiding strategy. Once again, this cyclic behaviour leads to a closed form for $u(i, \xi)$, with details provided in Appendix \ref{append:numerdetails_cyclic}.

\section{Conclusion}
\label{sec:conclude}
The history of the discrete search game studied in this paper dates back to the 1960s, when \cite{Bram1963} considers the special case of unit search times.
Despite its long history, it was not until recently that \cite{clarkson2022classical} shows the existence of an optimal search strategy in general.
This paper complements these earlier works by presenting an algorithm that computes an optimal strategy both for the hider and for the searcher to arbitrary accuracy. 
We further consider the practicalities of implementing the algorithm, and demonstrate its performance in a numerical study.

Our algorithm relies on generating a new search sequence in each iteration to strengthen the searcher's repertoire of pure search strategies.
The progression of the algorithm is analogous to that of a cutting plane method used in convex optimization (see, for example, Chapter 13 in \cite{luenberger2008linear}).
The same idea---constructing a finite strategy set for the infinite-strategy player using their own previous optimal responses---is applicable to other semi-finite two-person zero-sum games.

An important assumption that makes the discrete search game in this paper more tractable is that it takes no time for the searcher to travel between search locations---which is reasonable in some applications, such as search in the cyber world, or if the travel time is substantially smaller than the time spent searching, but not in all cases.
Further, if there are multiple hiders or searchers, then how would such a team coordinate to achieve their common goals?
Studying these extensions would require new formulation and new techniques.


\section*{Acknowledgments}
The authors are grateful for the support of the EPSRC funded EP/L015692/1 STOR-i Centre for Doctoral Training, and would like to thank Kevin Glazebrook and Dashi Singham for helpful discussions and comments.

\begin{appendices}

\section{An Upper Bound for the Expected Time to Detection} \label{append:numerdetails_acyclic}
For any Gittins search sequence $\xi$, a finite approximation to \eqref{eqn:genv_i(xi)} after any $R\in\{0,1,\ldots\}$ searches of box $i$ gives a lower bound for $u(i,\xi)$. To support the discussion in Section \ref{sec:numercalcs}, this section presents a method to compute an upper bound for $u(i,\xi)$.

Write
\begin{equation*}
l \equiv \floor[\Bigg]{\max_{i,j \in \{1,\ldots , n\}} \frac{\log(1-\alpha_i)}{\log(1-\alpha_j)}} + 1.
\end{equation*}
For any $i,j \in \{1,\ldots, n\}$, since $l > \log(1-\alpha_i)/\log(1-\alpha_j)$, then $(1-\alpha_j)^l < (1-\alpha_i)$, so between any two successive searches of box $i$, any Gittins search sequence will make at most $l$ searches of box $j$, for $i \neq j$.
Therefore, for any $i \in \{1,\ldots,n\}$, no more than time $\widehat{l} \equiv \sum_{j=1}^n lt_j$ can elapse between successive searches of box $i$ following any Gittins search sequence. 


It follows that, for any Gittins search sequence $\xi$ and $R \in \{0,1,\ldots\}$, following $\xi$ until $R$ searches of box $i$ have been made, then, after that, assuming box $i$ is searched at regular time intervals of length $\widehat{l}$ gives an upper bound on $u(i,\xi)$. In other words,
\begin{align} \label{eqn:upperboundV}
u(i,\xi) &\leq \tau_i(1,\xi)+\sum_{r=1}^{R} (1-\alpha_i)^{r} [\tau_i(r+1,\xi)-\tau_i(r,\xi)] + (1-\alpha_i)^{R} \sum_{r=1}^{\infty} \widehat{l}(1-\alpha_i)^{r}  \nonumber \\
&= \tau_i(1,\xi)+\sum_{r=1}^{R} (1-\alpha_i)^{r} [\tau_i(r+1,\xi)-\tau_i(r,\xi)] + \frac{\widehat{l}(1-\alpha_i)^{R+1}}{\alpha_i}.
\end{align} 
Note that the first two terms in \eqref{eqn:upperboundV} are the lower bound for $u(i,\xi)$ obtained via a finite approximation to \eqref{eqn:genv_i(xi)} after $R$ searches of box $i$. Therefore, we increase $R$ until the ratio of the third term of \eqref{eqn:upperboundV} divided by the first two terms of \eqref{eqn:upperboundV} is less than $10^{-10}$.

\section{The Expected Time to Detection in a Cyclic Game} \label{append:numerdetails_cyclic}
To support the discussion in Section \ref{sec:cyclic_games}, this section presents a method to compute the expected time to detection $u(i,\xi)$ in a cyclic game, when the hider hides in any box $i \in \{1,\ldots , n\}$ and the searcher plays any Gittins search sequence $\xi$ against any hiding strategy $\mathbf{p}$.


Any cyclic search game satisfies \eqref{eqn:cyclic}. Hence, as noted in \cite{Matula}, after an initial transient period, any Gittins search sequence $\xi$ will make $\widehat{x} \equiv \sum_{i=1}^n x_i$ consecutive searches involving exactly $x_i$ visits of box $i$ for $i=1,\ldots , n$. 
By \eqref{eqn:cyclic}, the posterior probabilities that the hider is in each box will be the same before and after these $\widehat{x}$ searches have been made. Therefore, $\xi$ will cycle these $\widehat{x}$ searches indefinitely; 
such repetition allows a closed form for $u(i,\xi)$, $i=1,\ldots , n$, to be calculated as follows.

Suppose $\xi$ has entered the cycle of $\widehat{x}$ searches after some $R$ searches of box $i$ have been made. 
Recall $\tau_i(r,\xi)$ is the time at which the $r$th search of box $i$ is made under $\xi$ for $r=1,2,\ldots$. Then, for any $a \in \{0,1,2,\ldots\}$ and $r \in \{1,2,\ldots\}$ , we have 
\begin{equation} \label{eqn:b_simplification}
\tau_i(R+r+ax_i, \xi) = \tau_i(R+r, \xi) + a\widehat{t},
\end{equation}
where $\widehat{t} \equiv \sum_{i=1}^n x_it_i$ is the time it takes to complete one cycle of $\widehat{x}$ searches. 

Write $u_R \equiv \sum_{r=1}^{R} (1-\alpha_i)^{r} [\tau_i(r+1,\xi)-\tau_i(r,\xi)]$. 
We may write \eqref{eqn:genv_i(xi)} as
\begin{equation*}
u(i,\xi) = \tau_i(1,\xi) + u_R + (1-\alpha_i)^R \sum_{a=0}^\infty \left( \sum_{r=1}^{x_i} (1-\alpha_i)^{ax_i+r} [\tau_i(R+r+ax_i+1,\xi)-\tau_i(R+r+ax_i,\xi)] \right).
\end{equation*}
It follows from \eqref{eqn:b_simplification} that we have
\begin{align*}
\frac{u(i,\xi)-\tau_i(1,\xi)-u_R}{(1-\alpha_i)^R} 
&= \sum_{a=0}^\infty (1-\alpha_i)^{ax_i} \left( \sum_{r=1}^{x_i} (1-\alpha_i)^r \left[\tau_i(R+r+1, \xi) - \tau_i(R+r, \xi)\right] \right)  \\
&=\frac{A_i(R,\xi)}{(1-(1-\alpha_i)^{x_i})},
\end{align*}
where 
\begin{equation*} 
A_i(R,\xi) \equiv \sum_{r=1}^{x_i} (1-\alpha_i)^{r} \left[\tau_i(R+r+1, \xi) - \tau_i(R+r, \xi)\right].
\end{equation*}
Therefore, to evaluate $u(i,\xi)$ precisely, we only need to calculate $\tau_i(r,\xi)$ for those $r \in \{1,\ldots, R+x_i+1\}$, which we discuss in the following.

As justified in Section \ref{sec:numercalcs}, for $\mathbf{p}_k$ a solution to the linear program in Step \ref{step:solveandcheck} of Algorithm \ref{al:minExpCost}, using floating-point numbers to compute the Gittins indices in \eqref{eq:simpleGI} will calculate $\tau_i(r,\xi)$ for $r \in \{1,2,\ldots\}$ for some $\xi \in \mathcal{C}^{\rm B}_{\mathbf{p}_k}$, which is sufficient for step \ref{step:update_bounds} of Algorithm \ref{al:minExpCost}. To determine $R$, we 
evaluate $\xi$ 
until $\widehat{x}$ consecutive searches involve $x_j$ searches of box $j$ for $j=1,\ldots, n$.

On the other hand, floating-point indices cannot reliably calculate $\tau_i(r,\xi)$ for \emph{specific} search sequences $\xi\in \widehat{\mathcal{C}}^{\text{B}}_{\mathbf{p}_0}$, a requirement of step \ref{step:initialize} of Algorithm \ref{al:minExpCost}. To do this, 
a set of alternative indices is derived below, which encodes integers rather than floating-point numbers.


First, note that, by \eqref{eqn:cyclic}, the first $\widehat{x}$ searches of any $\xi \in \widehat{\mathcal{C}}^{\text{B}}_{\mathbf{p}_0}$ will involve $x_j$ searches of box $j$, $j=1,\ldots , n$. Therefore, we may take $R=0$, so only need calculate $\tau_i(r,\xi)$ for $r \in \{1,\ldots, x_i+1\}$.
For any $\xi \in \widehat{\mathcal{C}}^{\text{B}}_{\mathbf{p}_0}$, all $n$ indices in \eqref{eq:simpleGI} are equal at the start of the search, say to $y$. 
For $i=1,\ldots , n$, suppose $r_i \in \{1,\ldots, x_i\}$ searches of box $i$ have been made, so the current corresponding index in \eqref{eq:simpleGI} is $y(1-\alpha_i)^{r_i}$. Then we have $$y(1-\alpha_i)^{r_i} \propto (1-\alpha_i)^{r_i} = \left((1-\alpha_i)^{x_i}\right)^{r_i/x_i}= c^{r_i/x_i},$$
where $c\equiv(1-\alpha_i)^{x_i} \in (0,1)$ is constant over all boxes by \eqref{eqn:cyclic}. 
Therefore, the rule in \eqref{eq:simpleGI} is equivalent to searching any box $j$ satisfying
\begin{equation} \label{eq:cyclicGI}
j = \argmax_{i \in \{1, \ldots ,  n\}} \frac{x_i}{r_i}.
\end{equation}
Yet, both $x_i$ and $r_i$ are integers, so, unlike using \eqref{eq:simpleGI}, ties will always be detected using \eqref{eq:cyclicGI}.

\end{appendices}

\bibliographystyle{apalike}
\bibliography{References}

\end{document}